\documentclass[12pt,reqno]{amsart}
\usepackage{amssymb,latexsym,amsmath,amsthm,color,dsfont}
\usepackage{mathrsfs,cite}
\usepackage{enumerate}
\usepackage{stmaryrd}
\usepackage{fullpage}
\usepackage{ulem}

\usepackage{graphicx,anysize,float,booktabs,geometry}
\usepackage{pdflscape}
\usepackage{caption} 
\usepackage{hyperref}

\hypersetup{citecolor=red, linkcolor=blue, colorlinks=true}

\allowdisplaybreaks

\marginsize{2cm}{2cm}{2cm}{0cm}
\theoremstyle{plain}
\newtheorem{theorem}{Theorem}[section]
\newtheorem{lemma}[theorem]{Lemma}
\newtheorem{definition}[theorem]{Definition}
\newtheorem{proposition}[theorem]{Proposition}

\newtheorem{conjecture}[theorem]{Conjecture}
\newtheorem{remark}[theorem]{Remark}

\numberwithin{equation}{section}

\newcommand{\mc}{\mathcal}

\newcommand{\M}{\mathcal{M}}
\newcommand{\Q}{\mathcal{Q}}
\newcommand{\orb}{\mathcal{O}}
\def\cA{\mathcal A}
\def\cC{\mathcal C}
\def\cL{\mathcal L}
\def\cM{\mathcal M}
\def\cO{\mathcal O}
\def\cQ{\mathcal Q}
\def\cS{\mathcal S}

\newcommand{\tr}{\mathrm{Tr}}
\newcommand{\Tr}{\mathrm{Tr}}

\newcommand{\PG}{\mathrm{PG}}

\newcommand{\GL}{\mathrm{GL}}
\newcommand{\PGL}{\mathrm{PGL}}

\newcommand\F{\mathbb{F}}

\def\w{\omega}

\def\<{\langle}
\def\>{\rangle}

\title{An infinite family of $m$-ovoids of the hyperbolic quadrics $\Q^+(7,q)$}
\author[Pavese, Zou]{Francesco Pavese, Hanlin Zou}

\address{Francesco Pavese, Department of Mechanics, Mathematics and Management, Polytechnic University of Bari, Via Orabona 4, 70125 Bari, Italy; }
\email{francesco.pavese@poliba.it;}

\address{Hanlin Zou, School of Mathematical Sciences, Zhejiang University, Hangzhou 310058, China}
\email{zouhanlin@zju.edu.cn}

\begin{document}
\maketitle
\begin{abstract}
An infinite family of $(q^2+q+1)$-ovoids of $\Q^+(7,q)$, $q\equiv 1\pmod{3}$, admitting the group $\PGL(3,q)$, is constructed. The main tool is the general theory of generalized hexagons.
\end{abstract}

\section{Introduction}\label{sec1}

Let $q$ be a prime power, $\F_q$ the finite field of order $q$, and $V$ a finite dimensional vector space over $\F_q$. Let $f$ be a non-degenerate reflexive sesquilinear form or a non-singular quadratic form on $V$. The finite classical polar space $\mc{S}$ associated with $(V, f)$ is the geometry consisting of the totally singular or totally isotropic subspaces with respect to $f$ of the ambient projective space $\PG(V)$, according to whether $f$ is a quadratic or sesquilinear form. 
The totally singular or totally isotropic one-dimensional subspaces are the {\it points} of $\mc{S}$, and the collection of all the points of $\mc{S}$ will be denoted by $\mc{P}$. The totally singular or totally isotropic subspaces of maximum dimension are called the {\it generators} of $\mc{S}$.
The {\it rank} of $\mc{S}$ is the vector space dimension of its generators. A finite classical polar space of rank 2 is a point-line geometry, and is also called a {\it finite generalized quadrangle}. Polar spaces over finite fields are very interesting geometric structures because they possess large automorphism groups, namely the finite classical groups. In this context it is natural to investigate combinatorial objects embedded in polar spaces admitting a fairly large group. Here, along these lines, we are interested in particular substructures of polar spaces called $m$-ovoids, having many symmetries.

The notion of an $m$-ovoid came from ``ovoids'' of the projective space $\PG(3,q)$. It was first defined in generalized quadrangles by Thas \cite{Thas89}, and was later extended naturally to finite classical polar spaces of higher rank in the work of Shult and Thas \cite{ST94}, where an {\it $m$-ovoid} was defined to be a set of points having exactly $m$ points in common with each generator (a 1-ovoid is simply called an ovoid). In \cite{BLP09} the authors found that $m$-ovoids and tight sets of generalized quadrangles share the property that they have two intersection numbers with respect to perps of points (one for points in the set and the other for points not in the set), and they coined such sets intriguing. Here, the perp of a point $P$ is the set of the points of the generalized quadrangle that are collinear with $P$. In a subsequent paper \cite{BKLP07}, the authors extended the concept of an intriguing set to finite classical polar spaces of higher rank. Over the past two decades, intriguing sets have been extensively studied and they are known to have close connection to many geometric/combinatorial objects, such as translation planes, strongly regular graphs, two-weight codes, Boolean degree one functions, completely regular codes of strength 0 and covering radius one, and Cameron-Liebler line classes. We refer the reader to \cite{BLMX,CP18,FI19,FMRQZ,FT20,FWX20,KNS,LN20} for some recent results on intriguing sets of polar spaces.

There are some straightforward methods for constructing $m$-ovoids. Let $\mc{S}$ be a polar space of rank $r$ over the field $\F_q$. As an example, the whole point set of $\mc{S}$ is a $\frac{q^r-1}{q-1}$-ovoid. Let $A$ and $B$ be an $m$-ovoid and $n$-ovoid of $\mc{S}$, respectively. If $A\subseteq B$, then $B\setminus A$ is an $(m-n)$-ovoid. In particular, the complement of $A$ is a $(\frac{q^r-1}{q-1}-m)$-ovoid. Dually, if $A$ and $B$ are disjoint, then $A\cup B$ is an $(m+n)$-ovoid.

In hyperbolic polar spaces $\mc{S}=\Q^+(2r-1,q)$, there are some less straightforward ways to construct $m$-ovoids. First of all, a non-degenerate hyperplane section of $\mc{S}$ which is an embedded $\Q(2r-2,q)$ is a $\frac{q^{r-1}-1}{q-1}$-ovoid, and an $m$-ovoid of the embedded $\Q(2r-2,q)$ is also an $m$-ovoid of $\mc{S}$. Furthermore, new $m$-ovoids can be obtained from old ones by derivation (see \cite{Kelly07}). There are other constructions in hyperbolic spaces of different ranks. If $r=2$, then $\mc{S}=\Q^+(3,q)$ which is a grid. It is easy to see that $\Q^+(3,q)$ can be partitioned into ovoids and so $m$-ovoids exist for all possible $m$. If $r=3$, then $\mc{S}=\Q^+(5,q)$ and an ovoid of $\Q^+(5,q)$ is equivalent to a spread of $\PG(3,q)$ under the Klein correspondence. It is known that the lines of $\PG(3,q)$ can be partitioned into spreads (such a partition is called a packing of $\PG(3,q)$) \cite{Den73}. This means that $m$-ovoids of $\Q^+(5,q)$ exist for all possible $m$. The situation is much more different when $r$ is greater than 3 and only a few results are known. We will be focusing on $\Q^+(7,q)$ in this work. It is known that ovoids exist in $\Q^+(7,q)$ when $q$ is even, an odd prime or $q\equiv $ 0 or $2\pmod{3}$ \cite[Table 7.3]{HT16}. As mentioned above, $m$-ovoids of $\Q(6,q)$ are $m$-ovoids of $\Q^+(7,q)$. Constructions of $m$-ovoids of $\Q(6,q)$ can be found in \cite{BKLP07,CP16,LN20}. There are further known $m$-ovoids of $\Q^+(7,q)$ which we present in Section \ref{concluding} when dealing with the isomorphism issue.

In this paper, we construct an infinite family of $(q^2+q+1)$-ovoids of $\Q^+(7,q)$, $q\equiv 1\pmod{3}$, admitting $\PGL(3,q)$ as an automorphism group and not contained in a hyperplane section. The promised $(q^2+q+1)$-ovoid will be described as a set $\cO$ of points covered by the planes spanned by two incident lines of a thin hexagon embedded in $\Q^+(7,q)$. The advantage is that in this way one can apply basic results from the theory of generalized hexagons. The properties of the hexagon that will be needed are achieved in Section \ref{sec_pre}.  In Section \ref{sec_main} we define the set $\cO$ precisely and show that it is a $(q^2+q+1)$-ovoid. Finally, we show that our example is new in Section \ref{concluding}, and a problem is posed in the end.

For the rest of this paper, we shall use the following definition of $m$-ovoids rather than the one given by Shult and Thas as mentioned in the second paragraph.
\begin{definition}\label{def_movoid}
\textup{Let $\mc{P}$ be the point set of $\Q^+(2r-1,q)$.
A subset $\mc{M}$ of $\mc{P}$ is called an {\it $m$-ovoid} if there is an integer $m>0$ such that for all $P\in \mc{P}$, 
\begin{equation}\label{def}
|P^\perp\cap \mc{M}|=\begin{cases}
m\theta_{r-1}-\theta_{r-1}+1,&\textup{if~} P\in \mc{M},\\
m\theta_{r-1},&\textup{if~}P\notin\mc{M},
\end{cases}
\end{equation}
where $\theta_r=q^{r-1}+1$ and $P^{\perp}$ is the set of points in $\Q^+(2r-1,q)$ that are collinear with $P$. }
\end{definition}

\section{Preliminaries}\label{sec_pre}

Let $p$ be a prime, and $q=p^e$, where $e\ge 1$ is an integer.  The multiplicative group of $\F_q$ will be denoted by $\F_q^*$. For an integer $n\ge 1$, the {\it relative trace} $\Tr_{q^n/q}$  from $\F_{q^n}$ to $\F_q$ is defined by
\[\Tr_{q^n/q}(z)=z+z^q+z^{q^2}+\cdots+z^{q^{n-1}},\; \forall z\in\F_{q^n}.\]
In particular, if $q=p$, then $\Tr_{q^n/q}$ is called the {\it absolute trace}.

\subsection{Cubic polynomials over $\F_q$}
Let $q=p^e$ be a prime power, where $p\ne 3$ is a prime. Let $f(x)=x^3+cx+d$ be a cubic polynomial  over $\F_q$, and let $\gamma_1,\gamma_2,\gamma_3$ be its roots in some extension field of $\F_q$. The discriminant of $f$ is defined by
\[
 \Delta_f:=(\gamma_1-\gamma_2)^2(\gamma_2-\gamma_3)^2(\gamma_3-\gamma_1)^2,
\]
which equals $-4c^3-27d^2$ for all $q$. In particular, when $q$ is even, we have $\Delta_f=d^2$. 

\begin{lemma}\label{-3s}
If $q\equiv 1\pmod{3}$ with $q$ odd, then $-3$ is a nonzero square in $\F_q$.
\end{lemma}
\begin{proof}
Let $q=p^e$. In the case $p\equiv 2 \pmod{3}$, we have $e$ is even, and so $-3$ is a nonzero square in $\F_q$. In the case $p\equiv 1\pmod{3}$, by the quadratic reciprocity law we have
\[\left(\frac{-3}{p}\right)=\left(\frac{-1}{p}\right)\left(\frac{3}{p}\right)=(-1)^{(p-1)/2}(-1)^{(3-1)/2\cdot(p-1)/2}\left(\frac{p}{3}\right)=1.\]
Here $(\frac{\cdot}{p})$ is the Legendre symbol. Hence $-3$ is a square in $\F_p$ and also a square in $\F_q$. 
\end{proof}

We shall need the following theorem giving the nonexistence condition for the roots of $f$ in $\F_q$ in various situations.

\begin{theorem}[\cite{Dixon,Williams1975}]\label{cubic}
Let $q$ be a prime power with $q\equiv 1\pmod{3}$. Suppose that  $f(x)=x^3+cx+d$ is a polynomial over $\F_q$ with discriminant $\Delta_f\ne 0$.
\begin{enumerate}
\item[(i)] If $q$ is odd, and write $-3=\alpha^2$ for some $\alpha\in\F_{q}$, then $f$ has no roots in $\F_q$ if $\Delta_f$ is a square in $\F_q$, say $\Delta_f=81\beta^2$ for some $\beta\in\F_q$, and $2^{-1}(-d+\alpha\beta)$ is not a cube in $\F_{q}$.

\item[(ii)] If $q$ is even, then $f$ has no roots in $\F_q$ if $\Tr_{q/2}(c^3d^{-2})=\Tr_{q/2}(1)$, and the roots $t_1,t_2$ of $t^2+dt+c^3=0$ are not cubes in $\F_q$.
 \end{enumerate}
\end{theorem}

\begin{lemma}\label{xyz}
Let $q$ be a prime power with $q\equiv 1\pmod{3}$ and $\theta$ a non-cubic element of $\F_q^*$. Then the equation
$x^3+\theta y^3+\theta^2z^3-3\theta xyz=0$ has a unique solution in $\F_q$, i.e., $x=y=z=0$. 
\end{lemma}
\begin{proof}
If $z=0$, then the equation is reduced to $x^3+\theta y^3=0$. Since $\theta$ is not a cube, the above equation has solutions in $\F_q$ if and only if $x=y=0$. Similarly, if $x=0$ or $y=0$, the original equation has solutions in $\F_q$ for $x,y,z$ if and only if $x=y=z=0$.
So it is enough to show that the polynomial \[f(x):=x^3-3\theta yzx+\theta y^3+\theta^2z^3\] has no roots in $\F_q$ for any $y,z \in \F_q^*$. 

{\it Case 1. Assume $q$ is odd.} The discriminant of $f$ is
\begin{align*}
\Delta_f=&-4(-3\theta yz)^3-27(\theta y^3+\theta^2z^3)^2=(-3)\cdot 9(\theta y^3-\theta^2z^3)^2.
\end{align*}
By Lemma \ref{-3s}, we see that $\Delta_f$ is a nonzero square in $\F_q$. Furthermore, if we write $-3=\alpha^2$, and $ \Delta_f=81\beta^2=-27\alpha^2\beta^2$, then $\alpha\beta=\pm(\theta y^3-\theta^2z^3)$. This implies that $2^{-1}(\alpha\beta-(\theta y^3+\theta^2z^3))=\theta y^3$ or $\theta^2z^3$, neither of which is a cube in $\F_q$. Therefore $f(x)$ has no solutions in $\F_q$ by Theorem \ref{cubic}.

{\it Case 2. Assume $q$ is even.} Then $q=2^n$ and $n$ is even. We have
\[\tr_{q/2}\left(\frac{(\theta yz)^3}{(\theta y^3+\theta^2z^3)^2}\right)=\tr_{q/2}\left(\frac{\theta y^3}{\theta y^3+\theta^2 z^3}+\frac{\theta^2y^6}{(\theta y^3+\theta^2 z^3)^2}\right)=0=\tr_{q/2}(1).\]
Moreover, the equation $t^2+(\theta y^3+\theta^2 z^3)t+(\theta y z)^3=0$ has solutions $t_1=\theta y^3$ and $t_2=\theta^2 z^3$, neither of which is a cube in $\F_q$. Using Theorem \ref{cubic} again, we conclude that $f(x)$ has no solutions in $\F_q$.
\end{proof}

\subsection{An embedded generalized hexagon in $\PG(7,q)$}

Let $\mc{I}$ be an point-line incidence geometry with point set $\mc{P}$, line set $\mc{L}$ and incidence relation $\textbf{I}\subseteq \mc{P}\times \mc{L}$. The {\it incidence graph} of $\mc{I}$ is the graph with vertex set $\mc{P}\cup\mc{L}$, where adjacency is given by the incidence relation \textbf{I}. A {\it generalized hexagon} is a point-line geometry such that its incidence graph has diameter 6 and girth 12. A generalized hexagon is said to have {\it order}
$(s,t)$ if every line is incident with exactly $s+1$ points and if every point is incident with precisely $t+1$ lines. A quick example of a generalized hexagon is the {\it flag geometry} $\mc{I}$ of a projective plane $\Pi$ of order $s$ defined as follows. The points of $\mc{I}$ are the {\it flags} of $\Pi$ (i.e., the incident point-line pairs); the lines of $\mc{I}$ are the points and lines of $\Pi$. Incidence between points and lines of $\mc{I}$ is reverse containment. It follows that $\mc{I}$ is a generalized hexagon of order $(s,1)$. We refer the reader to \cite{polygon} for more background information about generalized hexagons.

Let $\cS_{2,2}$ be the Segre variety of $\PG(8, q)$ consisting of the $(q^2+q+1)^2$ points represented by
\begin{align}
& (a_1b_1, a_1b_2, a_1b_3, a_2b_1, a_2b_2, a_2b_3, a_3b_1, a_3b_2, a_3b_3), \label{point}
\end{align}
where $a_i, b_i \in \F_q$, $i = 1,2,3$, with $(a_1, a_2, a_3) \ne (0,0,0)$ and $(b_1, b_2, b_3) \ne (0,0,0)$. By means of the map 
\begin{equation}
\rho: M = (m_{i j}) \in \cM_{3,3}(q) \mapsto (m_{11}, m_{12}, m_{13}, m_{21}, m_{22}, m_{23}, m_{31}, m_{32}, m_{33}) \in \F_q^9, \label{map}
\end{equation}
points of $\cS_{2,2}$ correspond to rank one $3 \times 3$ matrices over $\F_q$. Moreover, $\cS_{2, 2}$ is contained in the parabolic quadric $\cQ(8, q)$ given by
\begin{align*}
& X_2 X_4 - X_1 X_5 + X_3 X_7 - X_1 X_9 + X_6 X_8 - X_5 X_9 = 0.
\end{align*}
There is a group, say $G$, isomorphic to $\PGL(3, q)$ fixing both $\cS_{2, 2}$ and $\cQ(8, q)$. This can be seen by the map
\[\rho^{-1}(X)\mapsto A\rho^{-1}(X)A^{-1}, \;\forall A\in \GL(3,q), X\in \cS_{2, 2}\cap\cQ(8, q).\]
Such a group fixes the point of $\PG(8, q)$ corresponding to the $3 \times 3$ identity matrix and the hyperplane $\Pi: X_1+X_5+X_9 = 0$. Note that $\Pi \cap \cQ(8, q)$ is hyperbolic, elliptic or a cone, according to whether $q \equiv 1$, $-1$, or $0 \pmod{3}$. 

Assume that $q \equiv 1 \pmod{3}$ and denote by $\cQ^+(7, q)$ the hyperbolic quadric obtained by intersecting $\cQ(8, q)$ with $\Pi$. Let $\perp$ be the polarity of $\Pi$ associated with $\cQ^+(7, q)$. By \cite[pp. 99--100]{TV}, the set $\Pi \cap \cS_{2, 2}$ consists of $(q+1)(q^2+q+1)$ points and $2(q^2+q+1)$ lines of a {\it thin hexagon}, say $\Gamma$; the hexagon $\Gamma$ corresponds to the point-line flag geometry of $\PG(2, q)$. Every line of $\Gamma$ has $q+1$ points of $\Gamma$. Through a point $x$ of $\Gamma$ there pass two lines of $\Gamma$ spanning a plane and we will denote this plane by $\pi_x$. It is readily seen that $\pi_x \subset \cQ^+(7, q)$, for $x \in \Gamma$. The incidence graph of $\Gamma$ is the girth $12$ bipartite graph where the two parts are the points and the lines of $\Gamma$ and adjacency is given by incidence. The {\it distance} between points and lines of $\Gamma$ is that between two vertices of the incidence graph of $\Gamma$.

\begin{lemma}\label{l_0}
The set of $2(q+1)$ lines of $\Gamma$ having distance $1$ or $3$ from a point $x \in \Gamma$ span the hyperplane $x^\perp$ of $\Pi$.
\end{lemma}
\begin{proof}
By \cite[Lemma 1]{TV}, it is enough to show that these $2(q+1)$ lines are contained in the hyperplane $x^\perp$. To see this fact observe that if $r_1, r_2$ are the two lines of $\Gamma$ through $x$, then $r_1, r_2 \subset \pi_x \subset x^\perp$. If $\ell$ is a line of $\Gamma$ and $d(x, \ell) = 3$, then $\ell$ is incident with $r_1$ or $r_2$ in a point $z$ and $\ell \subset \pi_z \subset x^\perp$.
\end{proof}

It follows from the previous lemma that the lines of $\Gamma$ at distance $5$ from $x$ and the points of $\Gamma$ at distance $6$ from $x$ are not contained in $x^\perp$. 

Let $U_i$ be the point of $\PG(8, q)$ having $1$ in the $i$-th position and $0$ elsewhere.

\begin{lemma}\label{l_1}
If $x$ is a point of $\Gamma$, then $\pi_{x}^\perp$ meets $\Gamma$ precisely in the two lines of $\Gamma$ through $x$. 
\end{lemma}
\begin{proof}
Since the group $G$ acts transitively on points of $\Gamma$ (by Lemma \ref{lem_orb}), we may assume without loss of generality that $x$ is the point $U_3$. Then some easy calculations show that the $4$-dimensional projective space $X_4 = X_7 = X_8 = 0$ of $\Pi$, i.e. $\pi_{U_3}^\perp$, meets $\Gamma$ precisely in the two lines of $\Gamma$ through $U_3$. 
\end{proof}

An {\it apartment} of $\Gamma$ consists of $6q$ points and $6$ lines of $\Gamma$ forming an ordinary $6$-gon. By \cite[Lemma 2]{TV} an apartment of $\Gamma$ spans a $5$-dimensional projective space of $\Pi$.

\begin{lemma}\label{lem_cap}
If $x, y$ are distinct points of $\Gamma$, then $\pi_x \cap \pi_y$ is either a line of $\Gamma$ or a point of $\Gamma$ or empty, according to whether $d(x, y)$ equals $2$, $4$ or $6$, respectively. 
\end{lemma}
\begin{proof}
Let $x, y$ be distinct points of $\Gamma$ and let $\ell_1$, $\ell_2$ be the two lines of $\Gamma$ through $y$. If $d(x, y) = 2$, then $\pi_x \cap \pi_y = \langle x, y \rangle$. If $d(x, y) = 4$, then we may assume that $d(x, \ell_1) = 3$ and $d(x, \ell_2) = 5$. Hence there is a point $z \in \ell_1$ such that $d(x, z) = 2$. Since $\ell_2 \not\subset x^\perp$, it follows that $\pi_y \cap x^\perp = \ell_1$ and therefore $\pi_x \cap \pi_y = \{z\}$. If $d(x, y) = 6$, then $d(x, \ell_i) = 5$, $i = 1,2$, and hence there are four points $z_1, z_2, z_3, z_4$ of $\Gamma$ such that $z_i \in \ell_i$, with $d(x, z_i) = 4$, $i = 1,2$, and $d(x, z_3) = d(z_1, z_3) = d(x, z_4) = d(z_2, z_4) = 2$. The points $x, z_4, z_2, y, z_1, z_3$ are the vertices of an apartment of $\Gamma$. Since $\langle x, z_4, z_2, y, z_1, z_3 \rangle \simeq \PG(5, q)$, it follows that $|\pi_x \cap \pi_y| = 0$.
\end{proof}

\section{The construction}\label{sec_main}
Assume $q\equiv 1 \pmod{3}$ and take the notation introduced in the previous section. Set
\begin{align*}
& \cO = \bigcup_{x \in \Gamma} \pi_x .
\end{align*}
By Lemma \ref{lem_cap} if $x, y \in \Gamma$, $x \ne y$, and $|\pi_x \cap \pi_y| \ne 0$, then $\pi_x \cap \pi_y \subset \Gamma$. Hence $|\cO| = (q^2-q)(q+1)(q^2+q+1)+(q+1)(q^2+q+1) = (q^2+q+1)(q^3+1)$. This shows that $\cO$ has the correct size of a $(q^2+q+1)$-ovoid of $\Q^+(7,q)$ and we will show it is indeed a $(q^2+q+1)$-ovoid.

We first describe the orbits of the group $G$ on the point set of $\Q^+(7,q)$.

Let $\w$ be a primitive element of $\F_q$, i.e., $\F_q^*=\<\w\>$.

\begin{lemma}\label{lem_orb}
The group $G$ has $5$ orbits on points of $\cQ^+(7, q)$. 
\begin{enumerate}
\item The orbit $\cO_1$ consists of the $(q+1)(q^2+q+1)$ points of $\Gamma$; a representative for $\cO_1$ is $P_1 = U_3$ and $U_3^\perp$ contains exactly $2(q+1)$ lines of $\Gamma$.
\item The orbit $\cO_2$ consists of the $(q^3-q)(q^2+q+1)$ points of $\cO \setminus \Gamma$; a representative for $\cO_2$ is $P_2 = U_2+U_6$ and $P_2^\perp$ contains exactly $2$ lines of $\Gamma$.
\item The orbit $\cO_3$ has size $q^3(q+1)(q^2+q+1)/3$; a representative for $\cO_3$ is $P_3 = U_1 + \w^{\frac{q-1}{3}} U_5 + \w^{\frac{2(q-1)}{3}} U_9$ and $P_3^\perp$ contains exactly $6$ lines of $\Gamma$. 
\item The orbit $\cO_i$, $i = 4, 5$, has $q^3(q^2-1)(q-1)/3$ points; a representative for $\cO_4$ is $P_4 = U_2 + U_6 + \w U_7$ and for $\cO_5$ is $P_5 = U_3 + \w U_4 + \w U_8$; $P_i^\perp$, $i = 4, 5$, contains no line of $\Gamma$. 
\end{enumerate}
\end{lemma}
\begin{proof}
We shall find it helpful to work with the elements of $\PGL(3,q)$ as matrices in $\GL(3,q)$ and the points of $\PG(7,q)$ as $3\times 3$ matrices with trace zero.

\begin{enumerate}
\item Let $A\in \GL(3,q)$ and assume that $A\rho^{-1}(U_3)A^{-1}=\lambda \rho^{-1}(U_3)$ for some $\lambda\in\F_{q}^*$. This can be written explicitly as 
\[\left(\begin{matrix}
0&0& a_{11} \\
0 &0& a_{21} \\
0 &0& a_{31} 
\end{matrix}\right)=\lambda\left(\begin{matrix}
a_{31}& a_{32} &a_{33}\\
0&0&0\\
0&0&0
\end{matrix}\right),\]
which implies that $a_{11}=\lambda a_{33}$ and $a_{21}=a_{31}=a_{32}=0$. Since $A\in \GL(3,q)$, then $a_{22}a_{33}\neq 0$. It follows that the stabilizer of $P_1$ in $\GL(3,q)$ has size $q^3(q-1)^3$ and so
\[|\orb_1|=\frac{|\GL(3,q)|}{q^3(q-1)^3}=(q+1)(q^2+q+1).\]
The fact that $U_3^\perp$ contains $2(q+1)$ lines of $\Gamma$ follows from Lemma~\ref{l_0}. 
\item Let $A\in \GL(3,q)$ with $A\rho^{-1}(P_2)A^{-1}=\lambda \rho^{-1}(P_2)$ for some $\lambda\in\F_q^*$. Then $a_{21}=a_{31}=a_{32}=0$, $a_{12}=\lambda a_{23}$, $a_{11}=\lambda a_{22}$, $a_{22}=\lambda a_{33}$, and $a_{33}\neq 0$. Thus
\[|\orb_2|=\frac{|\GL(3,q)|}{q^2(q-1)^2}=q(q^3-1)(q+1).\]
The point $P_2$ belongs to $\pi_{U_3}$ and hence $P_2^\perp$ contains the two lines $\ell_1, \ell_2$ of $\Gamma$ through $U_3$. Let $\ell$ be a line of $\Gamma$. If $d(U_3, \ell) = 3$ and $\ell \subset P_2^\perp$, then $|\ell \cap \ell_1| = 1$ and $\langle \ell, \ell_1 \rangle$, $\langle P_2, \ell \rangle$, $\langle \ell_1, \ell_2 \rangle$ are three planes of $\cQ^+(7, q)$ spanning a solid and pairwise intersecting in a line. It follows that the solid is contained in $\cQ^+(7, q)$. Therefore $\ell \subset \pi_{U_3}^\perp$, contradicting Lemma~\ref{l_1}. If $d(U_3, \ell) = 5$ and $\ell \subset P_2^\perp$, then there is a line $\ell'$ incident with $\ell$ and with $\ell_1$. Hence $d(U_3, \ell') = 3$ and $\ell' \subset P_2^\perp$. Thus, as above, $\ell' \subset \pi_{U_3}^\perp$, contradicting Lemma~\ref{l_1}. 

\item Let $A\in \GL(3,q)$ with $A\rho^{-1}(P_3)A^{-1}=\lambda\rho^{-1}(P_3)$ for some $\lambda\in\F_q^*$. If $\lambda=1$, then $a_{12}=a_{13}=a_{21}=a_{23}=a_{31}=a_{32}=0$ and $a_{11}a_{22}a_{33}\neq 0$. If $\lambda=\w^{\frac{q-1}{3}}$, then $a_{11}=a_{13}=a_{21}=a_{22}=a_{32}=a_{33}=0$ and $a_{12}a_{23}a_{31}\neq 0$. If $\lambda=\w^{\frac{2(q-1)}{3}}$ then $a_{11}=a_{12}=a_{22}=a_{23}=a_{31}=a_{33}=0$ and $a_{13}a_{21}a_{32}\neq 0$. If $\lambda\notin\{1,\w^{\frac{q-1}{3}},\w^{\frac{2(q-1)}{3}}\}$, then $A=0$ which is not in $\GL(3,q)$. Therefore,
\[|\orb_3|=\frac{|\GL(3,q)|}{3(q-1)^3}=\frac{q^3(q+1)(q^2+q+1)}{3}.\]

The hyperplane $P_3^\perp$ of $\Pi$ contains the $5$-dimensional projective space $X_1 = X_5 = X_9 = 0$ which meets $\cQ^+(7, q)$ in a $\cQ^+(5, q)$ and intersects $\Gamma$ precisely in the apartment $\cA$ given by the $6$ lines $\langle U_2, U_3 \rangle, \langle U_3, U_6 \rangle, \langle U_6, U_4 \rangle, \langle U_4, U_7 \rangle, \langle U_7, U_8 \rangle, \langle U_8, U_2 \rangle$. A line $s$ of $\Gamma$ disjoint from $\cA$ cannot be contained in $P_3^\perp$, otherwise $s \cap \cQ^+(5, q)$ would be a point of $\Gamma$ not contained in $\cA$. Similarly, a line $s$ of $\Gamma$ incident with a line $r$ of $\cA$ cannot be contained in $P_3^\perp$, otherwise through $r$ there would be three planes of $\cQ^+(7, q)$ and necessarily two of these three planes would span a solid of $\cQ^+(7, q)$. In this case it follows that there exists a line $r'$ of $\cA$ incident with $r$ such that $s \subset \langle r, r' \rangle^\perp$, contradicting Lemma~\ref{l_1}. 
\item Assume $A\in \GL(3,q)$ with $A\rho^{-1}(P_4)A^{-1}=\lambda \rho^{-1}(P_4)$ for some $\lambda\in\F_{q}^*$. Then $\lambda^3=1$, $a_{11}=\lambda^2 a_{33}$, $a_{22}=\lambda a_{33}$, $a_{12}=\lambda a_{23}$, $a_{31}=\lambda^2 \w a_{23}$, $a_{21}=\lambda^2 \w  a_{13}$, and $a_{32}=\lambda \w  a_{13}$. With these conditions, we have
\[\det(A)=a_{33}^3+\w a_{23}^3+\w^2a_{13}^3-3\w a_{13}a_{23}a_{33}.\] Since $\det(A)\neq 0$,  by Lemma \ref{xyz}, the only restriction on $\{a_{13}, a_{23}, a_{33}\}$ is $(a_{13},a_{23},a_{33})\neq (0,0,0)$. Therefore,
\[|\orb_{4}|=\frac{|\GL(3,q)|}{3(q^3-1)}=\frac{q^3(q^2-1)(q-1)}{3}.\]
The computation for the orbit of $P_5$ is similar and we omit it. To see that $P_4$ and $P_5$ are not in the same orbit, assume $A\in\GL(3,q)$ such that $A\rho^{-1}(P_4)=\lambda \rho^{-1}(P_5)A$ for some $\lambda\in\F_q^*$. Then $\lambda^3\w=1$ which is impossible. 

We conclude that $G$ has five orbits as described above since $|\orb_1|+\cdots+|\orb_5|=(q^3+1)\frac{q^4-1}{q-1}$ which equals the number of points of $\Q^+(7,q)$.

The line $u$ joining $P_4$ and $P_5$ is secant to $\cQ^+(7, q)$ and $u^\perp$ is the $5$-dimensional projective space of $\Pi$ given by $\w X_3 + X_4 + X_8 = \w X_2 + \w X_6 + X_7 = 0$. Furthermore, $u^\perp$ is disjoint from $\Gamma$. Indeed, the point given in \eqref{point} belongs to $u^\perp$ if and only if 
\begin{align}
& a_1b_1 + a_2b_2 + a_3b_3 = 0 \nonumber \\
& \w a_1b_2 + \w a_2b_3 + a_3b_1 = 0 \label{sys} \\
& \w a_1b_3 + a_2b_1 + a_3b_2 = 0. \nonumber
\end{align}
The system \eqref{sys} has no non-trivial solutions since 
\begin{align*}
& \det \begin{pmatrix}
b_1 & b_2 & b_3 \\
\w b_2 & \w b_3 & b_1 \\
\w b_3 & b_1 & b_2 \\
\end{pmatrix} = 3 \w b_1 b_2 b_3 - b_1^3 - \w b_2^3 - \w^2 b_3^3 = 0  ,
\end{align*}
which implies that $(b_1, b_2, b_3) = (0,0,0)$ by Lemma \ref{xyz}. It follows that no line of $\Gamma$ is contained in $P_4^\perp$ or in $P_5^\perp$.
\end{enumerate}
\end{proof}

\begin{remark}\label{remark}
\textup{By Lemma~\ref{lem_orb} and \eqref{map}, the set $\cO$ consists precisely of the points of $\Pi$ corresponding to $3\times 3$ matrices over $\F_q$ having rank at most $2$. }
\end{remark}

We are now ready to prove the main theorem.
\begin{theorem}\label{main}
The set $\cO$ is a $(q^2+q+1)$-ovoid of $\cQ^+(7, q)$.
\end{theorem}
\begin{proof}
We show that $|P_i^\perp \cap \cO|$ equals $q^4+q^3+q^2+q+1$ if $i = 1,2$ or $(q^2+1)(q^2+q+1)$ if $i = 3,4,5$. 

The hyperplane $P_1^\perp$ of $\Pi$ contains $2(q+1)$ lines of $\Gamma$ and hence 
\begin{align*}
& |P_1^\perp \cap \Gamma| = 2q^2+2q+1. 
\end{align*}
Let $x$ be a point of $\Gamma$. If $x = P_1$ or $d(x, P_1) = 2$, then $\pi_{x} \subset P_1^\perp$. If $d(x, P_1) = 4$, then $\pi_x \cap P_1^\perp$ is a line of $\Gamma$ contained in $P_1^\perp$, whereas if $d(x, P_1) = 6$, then $\pi_x \cap P_1^\perp$ is a line with two points of $\Gamma$. Hence 
\begin{align*}
& |P_1^\perp \cap (\cO \setminus \Gamma)| = (2q+1) \times (q^2-q) + q^3 \times (q-1) = q^4 + q^3 - q^2 - q
\end{align*}
and therefore $|P_1^\perp \cap \cO| = q^4+q^3+q^2+q+1$.

Assume that the hyperplane $P_2^\perp$ of $\Pi$ contains the two lines of $\Gamma$ contained in $\pi_y$. It follows that 
\begin{align*}
& |P_2^\perp \cap \Gamma| = (q+1)^2. 
\end{align*}
Let $x$ be a point of $\Gamma$. If $x = y$, then $\pi_{x} \subset P_2^\perp$. If $d(x, y) = 2$, then $\pi_x \cap P_2^\perp$ is a line of $\Gamma$ contained in $P_2^\perp$. If $d(x, y) \ge 4$, then either $x \in P_2^\perp$ and $\pi_x \cap P_2^\perp$ is a line with one point of $\Gamma$, or $x \notin P_2^\perp$ and $\pi_x \cap P_2^\perp$ is a line with two points of $\Gamma$. Hence 
\begin{align*}
& |P_2^\perp \cap (\cO \setminus \Gamma)| = (q^2-q) + q^2 \times q + (q^3 + q^2)  \times (q-1) = q^4 + q^3 - q
\end{align*}
and therefore $|P_2^\perp \cap \cO| = q^4+q^3+q^2+q+1$.

If the hyperplane $P_3^\perp$ of $\Pi$ contains the apartment $\cA$ of $\Gamma$, then 
\begin{align*}
& |P_3^\perp \cap \Gamma| = (q-1)^2 + 6q = q^2+4q+1.
\end{align*}
Indeed there are exactly $2(q-1)^2$ lines of $\Gamma$ intersecting $P_3^\perp$ in a point not belonging to $\cA$.
Let $x$ be a point of $\Gamma$. If $x \in \cA$, then either $x$ is a vertex of $\cA$ and $\pi_{x} \subset P_3^\perp$ or $x$ is not a vertex of $\cA$ and $\pi_x \cap P_3^\perp$ is a line of $\cA$ that hence is contained in $P_3^\perp$. If $x \notin \cA$, then either $x \in P_3^\perp$ and $\pi_x \cap P_3^\perp$ is a line with one point of $\Gamma$, or $x \notin P_3^\perp$ and $\pi_x \cap P_3^\perp$ is a line with two points of $\Gamma$. Hence 
\begin{align*}
& |P_3^\perp \cap (\cO \setminus \Gamma)| = 6 \times (q^2-q) + (q-1)^2 \times q + (q^3 + q^2 - 2q)  \times (q-1) = q^4 + q^3 + q^2 - 3q
\end{align*}
and therefore $|P_3^\perp \cap \cO| = (q^2+1)(q^2+q+1)$.

If $i \in \{4, 5\}$, then no line of $\Gamma$ is contained in the hyperplane $P_i^\perp$ of $\Pi$. Hence 
\begin{align*}
& |P_i^\perp \cap \Gamma| = q^2+q+1.
\end{align*}
Let $x$ be a point of $\Gamma$. In this case either $x \in P_i^\perp$ and $\pi_x \cap P_i^\perp$ is a line with one point of $\Gamma$, or $x \notin P_i^\perp$ and $\pi_x \cap P_i^\perp$ is a line with two points of $\Gamma$. Hence 
\begin{align*}
& |P_i^\perp \cap (\cO \setminus \Gamma)| = (q^2+q+1) \times q + (q^3 + q^2 + q) \times (q-1) = q^4 + q^3 + q^2
\end{align*}
and therefore $|P_i^\perp \cap \cO| = (q^2+1)(q^2+q+1)$.
\end{proof}

\section{Concluding remarks}\label{concluding}
\subsection{Other known examples of $m$-ovoids of $\Q^+(7,q)$}

To the best of our knowledge, there are further known $m$-ovoids of $\cQ^+(7, q)$ not mentioned in Section \ref{sec1}.

\begin{itemize}
\item $m$-ovoids, $m \in \{q^2+q+1, q+1\}$, of $\cQ^+(7, q)$ arising from \cite[Construction 5.1]{CPS}: let $\cC$ be a cone of a $\cQ(6, q) \subset \cQ^+(7, q)$ such that $\cC \subset \cQ(6, q)$ and $\cC$ has either as vertex a point of $\cQ(6, q)$ and base a $\cQ(4, q)$ or as vertex a line and base a $\cQ(2, q)$. Let $\cC'$ be a cone having as vertex the same vertex of $\cC$ and as base either a $\cQ(4, q)$ or a $\cQ(2, q)$ and such that $\cC' \subset (\cQ^+(7, q) \setminus \cQ(6, q))$. Then $(\cQ(6, q) \setminus \cC) \cup \cC'$ is a $(q^2+q+1)$-ovoid of $\cQ^+(7, q)$. A line of $\cQ^+(7, q)$ meets $(\cQ(6, q) \setminus \cC) \cup \cC'$ in $0, 1, 2, q$ or $q+1$ points. By replacing $\cQ(6, q)$ with $\cQ^-(5, q)$, a similar construction yields $(q+1)$-ovoids of $\cQ^+(7, q)$.   

\item $m$-ovoids of $\cQ^+(7, q)$, $1 \le m \le q+1$, obtained by applying a group of order $q+1$ to an ovoid of $\cQ^+(7, q)$: it can be deduced from \cite{D}, that there is a group of order $q+1$, say $S$, fixing $\cQ^+(7, q)$ and a set $\cL$ consisting of $(q^2+1)(q^3+1)$ pairwise disjoint lines of $\cQ^+(7, q)$. In particular $\cL$ is a line-spread of $\cQ^+(7, q)$ and $\cS$ stabilizes each of the members of $\cL$ permuting in a single orbit the $q+1$ points of each of the lines of $\cL$. Therefore, if $\cO$ is an ovoid of $\cQ^+(7, q)$ then it turns out that $\cO^S$ is a set of $q+1$ pairwise disjoint ovoids of $\cQ^+(7, q)$. 
\end{itemize}

\subsection{The isomorphism issue}
We have constructed a $(q^2+q+1)$-ovoid $\cO$ in $\Q^+(7,q)$ for $q\equiv 1\pmod{3}$. As mentioned in Section \ref{sec1}, a non-degenerate hyperplane intersects $\Q^+(7,q)$ in a parabolic quadric $\Q(6,q)$ which is also a $(q^2+q+1)$-ovoid of $\Q^+(7,q)$. One may ask whether $\M$ is contained in a non-degenerate hyperplane. We explain below that the answer is no.

\begin{proposition}\label{prop_int}
There are lines of $\cQ^+(7, q)$ intersecting $\cO$ in exactly three points.
\end{proposition}
\begin{proof}
Consider the plane $\sigma$ spanned by $U_2$, $U_6$, $U_7$. Then $\sigma$ is a plane of $\cQ^+(7, q)$ and by Remark~\ref{remark}, $\sigma \cap \cO$ consists of the $3q$ points of the three lines joining the three points $U_2$, $U_6$, $U_7$. Hence a line of $\sigma$ not containing $U_2$, $U_6$, $U_7$ meets $\cO$ in three points. 
\end{proof}

Note that $U_1-U_2+U_4-U_5, U_6-U_5+U_9-U_8, U_2, U_3, U_4, U_6, U_7, U_8\in\cO$ and they span $\PG(7,q)$. It follows that $\cO$ is not contained in a hyperplane. Furthermore, we see from Proposition \ref{prop_int} that $\cO$ is not isomorphic to $(\cQ(6, q) \setminus \cC) \cup \cC'$ defined in the previous subsection. Therefore the $(q^2+q+1)$-ovoid $\cO$ is not equivalent to any known $m$-ovoids of $\Q^+(7,q)$.

\subsection{$m$-ovoids in $\Q(8,q)$}
Consider the parabolic quadric $\Q(8,q)$ given by
\[X_2 X_4 - X_1 X_5 + X_3 X_7 - X_1 X_9 + X_6 X_8 - X_5 X_9=0.\]
There is a group isomorphic to $\PGL(3,q)$ fixing $\Q(8,q)$ as mentioned in Section \ref{sec_pre}. 
With the aid of Magma \cite{Magma}, we found $(q^2+q+1)$-ovoids of $\Q(8,q)$ for $q=2,3,4,5$. When $q=2$, the $(q^2+q+1)$-ovoid found by the computer is actually an embedded $\Q^-(7,q)$. This is not interesting. However, for $q=3,4,5$, a Magma computation shows that there exists a $(q^2+q+1)$-ovoid not contained in a hyperplane and we conjecture that this is true for all larger $q$. 
\begin{conjecture}
There exists a $(q^2+q+1)$-ovoid of $\Q(8,q)$ which is not contained in a hyperplane for any prime power $q>2$.
\end{conjecture}

\bibliographystyle{plain}

\begin{thebibliography}{77}

\bibitem{BKLP07} J. Bamberg, S. Kelly, M. Law, T. Penttila, Tight sets and $m$-ovoids of finite polar spaces, {\it J. Combin. Theory, Ser. A} {\bf 114} (2007), 1293--1314.

\bibitem{BLP09}J. Bamberg, M. Law, T. Penttila, Tight sets and m-ovoids of generalized quadrangles, {\it Combinatorica} {\bf 29} (2009), 1--17.

\bibitem{BLMX}J. Bamberg, M. Lee, K. Melissa, Q. Xiang, A new infinite family of hemisystems of the Hermitian surface, {\it Combinatorica} {\bf 38} (2018), 43--66.

\bibitem{Magma}
W. Bosma, J. Cannon, C. Fieker, A. Steel, {\it Handbook of Magma Functions}, 2017.

 
\bibitem{CP16}A. Cossidente, F. Pavese, Hemisystems of $\Q(6,q)$, $q$ odd, {\it J. Combin. Theory Ser. A} {\bf 140} (2016), 112--122.

\bibitem{CP18}A. Cossidente, F. Pavese, On intriguing sets of finite symplectic spaces, {\it Des. Codes Cryptogr.} {\bf 86} (2018), 1161--1174.


\bibitem{CPS} D. Crnkovi\'{c}, F. Pavese, A. \v{S}vob, Intriguing sets of strongly regular graphs and their related structures, to appear in {\it Contrib. Discrete Math.}

\bibitem{Den73} R.H.F. Denniston, Packings of $\PG(3,q)$, in: {\it Finite geometric structures and their applications} (Centro Internaz. Mat. Estivo (C.I.M.E.), II Ciclo, Bressanone, 1972), Edizioni Cremonese, Rome, 1973, pp. 193--199. 

\bibitem{Dixon} L. E. Dickson, Criteria for the irreducibility of functions in a finite fields, {\it Bull. Amer. Math. Soc.} {\bf 13} (1906), 1--8.


\bibitem{D} R.H. Dye, Maximal subgroups of finite orthogonal groups stabilizing spreads of lines, {\it J. London Math. Soc. (2)}, {\bf 33} (1986), 279--293.

\bibitem{FMRQZ}T. Feng, K. Momihara, M. Rodgers, Q. Xiang, H. Zou, Cameron-Liebler line classes with parameter $x=\frac{(q+1)^2}{3}$, {\it Adv. Math.} {\bf 385} (2021), 107780.

\bibitem{FT20}T. Feng, R. Tao, An infinite family of $m$-ovoids of $\Q(4,q)$, {\it Finite Fields Appl.} {\bf 63} (2020), 101644.

\bibitem{FWX20}T. Feng, Y. Wang, Q. Xiang, On $m$-ovoids of symplectic polar spaces, {\it J. Combin. Theory Ser. A} {\bf 175} (2020), 105279.


\bibitem{FI19}Y. Filmus, F. Ihringer, Boolean degree 1 functions on some classical association schemes, {\it J. Combin. Theory Ser. A} {\bf162} (2019), 241--270.

\bibitem{HT16} J.W.P. Hirschfeld, J.A. Thas, {\it General Galois Geometries}, Springer Monographs in Mathematics, Springer, London, 2016.

\bibitem{Kelly07} S. Kelly, 
Constructions of intriguing sets of polar spaces from field reduction and derivation,
{\it Des. Codes Cryptogr.} {\bf 43} (2007), 1--8.



\bibitem{KNS}G. Korchm\'{a}ros, G.P. Nagy, P. Speziali, Hemisystems of the Hermitian surface, {\it J. Comb. Theory, Ser. A} {\bf 165} (2019), 408--439.

\bibitem{LN20} J. Lansdown, A.C. Niemeyer, A family of hemisystems on the parabolic quadrics, {\it J. Combin. Theory Ser. A} {\bf 175} (2020), 105280.


\bibitem{ST94}E.E. Shult, J.A. Thas, $m$-systems of polar space, {\it J. Combin. Theory Ser. A} {\bf 68} (1994), 184--204.


\bibitem{Thas89}J.A. Thas, Interesting pointsets in generalized quadrangles and partial geometries, {\it Linear Algebra Appl.} {\bf 114/115} (1989), 103--131.


\bibitem{TV} J.A. Thas, H. Van Maldeghem, On embeddings of the flag geometries of projective planes in finite projective spaces, {\it Des. Codes Cryptogr.}, {\bf 17} (1999), 97--104.

\bibitem{polygon}H. Van Maldeghem, {\it Generalized Polygons}, Birkh\"{a}user, Basel, 1998.
 
\bibitem{Williams1975}K. S. Williams, Note on cubics over $GF(2^n)$ and $GF(3^n)$, {\it J. Number Theory} {\bf 7} (1975), 361--365.


\end{thebibliography}

\end{document}